\documentclass[12pt,twoside]{amsart}
\usepackage{amssymb}
\usepackage{graphicx}
\usepackage[margin=1in]{geometry}

\newtheorem{theorem}{Theorem}[section]
\newtheorem{lemma}[theorem]{Lemma}

\theoremstyle{definition}

\newtheorem{example}[theorem]{Example}

\usepackage{latexsym}
\usepackage{color,hyperref}
\definecolor{darkblue}{rgb}{0,0,0.6}
\hypersetup{colorlinks,breaklinks,linkcolor=darkblue,urlcolor=darkblue,anchorcolor=darkblue,citecolor=darkblue}

\title[Two unfortunate properties of pure $f$-vectors]{Two unfortunate properties of pure $f$-vectors}

\author{Adri\'an Pastine}
\address{Department of Mathematical  Sciences\\ Michigan Tech\\ Houghton, MI  49931-1295}
\email{agpastin@mtu.edu; zanello@math.mit.edu}

\author{Fabrizio Zanello}

\thanks{2010 {\em Mathematics Subject Classification.} Primary: 05E40; Secondary: 13F55, 05E45, 05B07, 13H10.\\\indent
{\em Key words and phrases.} Pure simplicial complex; Cohen-Macaulay complex; Unimodality; Pure $f$-vector; Interval Property; Pure $O$-sequence; Steiner system}

\begin{document}

\begin{abstract}
The set of $f$-vectors of pure simplicial complexes is an important but little understood object in combinatorics and combinatorial commutative algebra. Unfortunately, its explicit characterization appears to be a virtually intractable problem, and its structure very irregular and complicated. The purpose of this note, where we combine a few different algebraic and combinatorial techniques, is to lend some further evidence to this fact.

We  first show that pure  (in fact, Cohen-Macaulay) $f$-vectors can be nonunimodal with arbitrarily many peaks, thus improving the corresponding results known for level Hilbert functions and pure $O$-sequences. We provide both an algebraic and a combinatorial argument for this result. Then, answering negatively a question of the second author and collaborators posed in the recent AMS Memoir on pure $O$-sequences, we show that the Interval Property fails for the set of pure $f$-vectors, even in dimension 2. 
\end{abstract}

\maketitle

\section{Introduction and preliminary results}

Simplicial complexes are important objects in a number of mathematical areas, ranging from combinatorics to algebra to topology (see e.g. \cite{BH,MS,St3}). Similarly to Macaulay's theorem for arbitrary $O$-sequences,  there exists a nice characterization of the $f$-vectors of arbitrary simplicial complexes, namely the {Kruskal-Katona theorem} \cite{St3}. However, beyond what we know for pure $O$-sequences,   little is known today about the structure of the $f$-vectors of \emph{pure} simplicial complexes, i.e., those complexes whose maximal faces are equidimensional. (See the last chapter of the recent monograph \cite{BMMNZ} for some initial results.)

The problem of characterizing pure $f$-vectors is considered to be nearly intractable. In fact, the goal of this note is to offer some strong, further evidence to the common perception that, quite unfortunately, the structure of pure $f$-vectors is extremely irregular and complicated. Our two main results are (see below for the  relevant definitions): 1) Pure $f$-vectors can be {nonunimodal with arbitrarily many peaks}. Our theorem improves the corresponding results recently shown for level Hilbert functions \cite{Za0} and pure $O$-sequences \cite{BMMNZ}, and we will give two different proofs of it:  one is algebraic, and holds for the strict subset of Cohen-Macaulay $f$-vectors, and the other is combinatorial. 2) Pure $f$-vectors fail the Interval Property (IP), even in dimension 2. The IP is a natural structural property known to hold for several sequences of interest in combinatorial commutative algebra (see \cite{BMMNZ,Za1}). It was originally conjectured by the second author \cite{Za1} for level Hilbert functions, where it is still open; it was  then conjectured in \cite{BMMNZ} (see also \cite{MNZ_exp})  for pure $O$-sequences, and asked as a question for pure $f$-vectors. However, recently, it was disproved by A. Constantinescu and M. Varbaro \cite{CV} for pure $O$-sequences of large degree, and by R. Stanley and the second author \cite{RSFZ} for arbitrary $r$-differential posets. 

Let $V$ be a finite set. A collection $\Delta$ of subsets of $V$ is called a \emph{simplicial complex} if it is closed under  inclusion, i.e., if for each $F\in\Delta$ and $G\subseteq F$, we have $G\in\Delta$. The elements of a simplicial complex $\Delta$ are its \emph{faces}, and the maximal faces are called \emph{facets}. The \emph{dimension of a face} is its cardinality minus 1, and the \emph{dimension of $\Delta$} is the largest of the dimensions of its faces. The complex $\Delta$ is \emph{pure} if all of its facets have the same dimension. Finally,  $\Delta$ is \emph{Cohen-Macaulay} if its associated \emph{Stanley-Reisner ring} is Cohen-Macaulay  (see  \cite{St3} for all details). It is a well-known and easy-to-prove fact of algebraic combinatorics that Cohen-Macaulay   complexes are a (very special) class of pure complexes.

Let $f=(f_{-1}=1,f_0,\dots ,f_e)$  be the \emph{$f$-vector} of an $e$-dimensional simplicial complex $\Delta$; i.e., $f_i$ counts the number of faces of $\Delta$ of dimension $i$. We say that $f$ is \emph{pure} if it is the $f$-vector of some pure simplicial complex. Let $h=(h_0=1,h_1,\dots ,h_{e+1})$ be the \emph{$h$-vector} of  $\Delta$; one way to define $h$ is as the (integer) vector whose entries are determined by those of $f$ via the following linear transformation:
\begin{equation}\label{hf}
f_{k-1}=\sum_{i=0}^k \binom{e+1-i}{k-i}h_i,
\end{equation}
for all  $k=0,1,\dots ,e+1$ (see e.g. \cite{St3}).

A vector $v$ is \emph{unimodal} if $v$ does not strictly increase  after a strict decrease. We say that $v$ is \emph{nonunimodal with $N$ peaks} if $N\ge 2$ is the number of (nonconsecutive) maxima of $v$. E.g., the vector $(1,15,22,18,20,20,15,22)$ is nonunimodal with 3 peaks.

Let $n$ and $i$ be positive integers. The \emph{i-binomial expansion of n} is defined as
$$n_{(i)}=\binom{n_i}{i}+\binom{n_{i-1}}{i-1}+\dots+\binom{n_j}{j},$$
for integers $n_i>n_{i-1}>\dots >n_j\ge j\ge 1$. It is easy to see that the $i$-binomial expansion of $n$ exists and is unique, for every $n$ and $i$  (e.g., see \cite{BH}, Lemma 4.2.6). Further, define $$n^{<i>}={n_i+1\choose i+1}+{n_{i-1}+1\choose i-1+1}+\dots +{n_j+1\choose j+1}.$$

A   sequence  $1,h_1,\dots,h_t$ of nonnegative integers is an \emph{$O$-sequence} if it satisfies \emph{Macaulay's growth condition}: $h_{i+1}\leq h_i^{<i>},$ for all indices $i=1,\dots,t-1$.

The following crucial result, essentially due to F.H.S. Macaulay and R. Stanley, characterizes the $h$-vectors of Cohen-Macaulay  complexes.

\begin{lemma}\label{stan} A vector $h$ is the $h$-vector of a Cohen-Macaulay simplicial complex if and only if it is an $O$-sequence.
\end{lemma}

\begin{proof} That any Cohen-Macaulay $h$-vector is an $O$-sequence is a standard commutative algebra result, which follows from a well-known theorem of Macaulay (see \cite{Ma} and, e.g., \cite{BH}, Theorem 4.2.10). The other implication is due to Stanley (\cite{St1}, Theorem 6).
\end{proof}

\section{Nonunimodality with arbitrarily many peaks}

The object of this  section is to show that Cohen-Macaulay $f$-vectors  can fail unimodality in an ``arbitrarily bad fashion'': namely, we  prove in Theorem \ref{main} that, for any $N\geq 2$, there exists a Cohen-Macaulay $f$-vector having exactly $N$ peaks.

The first example of a nonunimodal Cohen-Macaulay $f$-vector (with 2 peaks) was given by R. Stanley \cite{St1}, as a consequence of his characterization of the $h$-vectors of Cohen-Macaulay complexes (see Lemma \ref{stan}). Since Cohen-Macaulay   complexes are {pure}, any Cohen-Macaulay $f$-vector is a pure $f$-vector. In turn, pure $f$-vectors are the \emph{pure $O$-sequences} generated by \emph{squarefree} monomials, and pure $O$-sequences are those  \emph{level Hilbert functions}  corresponding to standard graded   artinian level  algebras  presented by {monomials} (see, e.g., \cite{BMMNZ,St1} for all  relevant definitions).  Thus, Theorem \ref{main} considerably extends the corresponding results proved in \cite{Za0} and \cite{BMMNZ} --- where nonunimodality with an arbitrary number of peaks was first shown  for arbitrary level Hilbert functions and then also for pure $O$-sequences --- by settling at once the case  of pure and Cohen-Macaulay $f$-vectors. 

Notice, on the other hand,  that for instance \emph{matroid} $f$-vectors, which are a subset of Cohen-Macaulay $f$-vectors, are conjectured to be all unimodal. (For some recent major progress, see  J. Huh  \cite{Hu}, and J. Huh and E. Katz \cite{HK};  as a consequence of their work,   this conjecture is now known to be true for all matroids representable over a field \cite{huh2,lenz}.) Thus,  Cohen-Macaulay $f$-vectors appear to be one of the smallest ``nice'' classes of level Hilbert functions where it is reasonable to expect the result of  Theorem \ref{main} to hold.

At the end of this section, we will give a second and entirely combinatorial proof of our theorem in the case of pure $f$-vectors.

\begin{theorem}\label{main} For any integer $N\ge 2$, there exists a nonunimodal Cohen-Macaulay $f$-vector having exactly $N$ peaks.
\end{theorem}

\begin{proof}
We want to construct a Cohen-Macaulay $f$-vector $f=(f_{-1}=1,f_0,f_1,\dots ,f_e)$ with exactly $N$ peaks, where $N\ge 2$ is fixed. The basic idea of the proof is the following. First of all, notice that because of formula (\ref{hf}) and Lemma \ref{stan}, each $f_j$ can be written as a linear combination of the first $j+2$ entries of an $O$-sequence $h=(h_0=1,h_1,h_2,\dots ,h_{e+1})$, where the last entry of this sum, $h_{j+1}$,  appears with coefficient 1, and it does not appear in the formula for any of the previous $f_i$. Further, $h_{j+1}$  appears in formula (\ref{hf}) for all of $f_j, f_{j+1}, \dots, f_e$, where it has a coefficient of $\binom{e-j}{0}, \binom{e-j}{1},\dots,\binom{e-j}{e-j}$, respectively. 

It is well known that for any fixed $e-j$ even, the binomial coefficients $\binom{e-j}{i}$ form a unimodal sequence with unique largest value $\binom{e-j}{(e-j)/2}$. Therefore, one moment's thought gives that, if we can pick $h_{j+1}$ to have a greater order of magnitude than that of all of the previous $h_i$, and also of
$$h_{j+2},h_{j+3}, \dots, h_{j+(e-j)/2+1}, h_{j+(e-j)/2+2},$$
then, asymptotically, $f$ has a peak corresponding to  $f_{j+(e-j)/2}$, i.e., $f_{(e+j)/2}$. 

Thus, we will begin this construction with $j=0$, and fix $e$ large enough (as a function of $N$) so to be able to suitably iterate the construction as many times as necessary. At that point, provided  we also choose all of the $h_i$ consistently with Macaulay's growth condition,  our  $f$ will be a nonunimodal Cohen-Macaulay $f$-vector with $N$ peaks, as we wanted to show. 

Fix $e=3\cdot 2^N-4$. We will construct $f$ having exactly $N$ peaks, with the $t$-th peak given by
$$f_{3(2^t-1)2^{N-t}-2},$$
for $t=1,2,\dots, N$.


Set, for instance, $h_1=n^{1+\epsilon}$, and
$$h_2=h_3=\dots =h_{e/2+2}=n,$$
where we will choose $n$ and $\epsilon >0$ at the very end so as to force all  entries of $h$ to simultaneously satisfy  Macaulay's growth condition. 
The coefficient of $h_1$ in formula (\ref{hf}) for $f_{i}$ is $\binom{e}{i}$, for any $i=0,1,\dots, e$. Thus, since the maximum of the binomial coefficients $\binom{e}{i}$ is achieved for $i=e/2$ and the order of magnitude of $h_1$ is largest, we obtain that for $n$ large (with respect to $e$,  hence $N$),
$$1<f_0<\dots <f_{e/2}>f_{e/2+1}.$$
(Notice that the inequalities $1<f_0<\dots <f_{e/2}$ are in fact always true, and easy to verify directly using formula (\ref{hf}), for any choice of the $h_i$.) Therefore,  the first peak of $f$ is given by $f_{e/2}$, i.e., $f_{3\cdot 2^{N-1}-2}$, as  desired. 



Proceeding in a similar fashion, after we have constructed the $(t-1)$-st peak, which corresponds to $f_{3(2^{t-1}-1)2^{N-t+1}-2},$
we  construct the $t$-th peak, for any $2\le t\le N$. 

Notice that $h_{3(2^{t-1}-1)2^{N-t+1}+1}$ is the entry  in formula (\ref{hf}) for $f_{3(2^{t-1}-1)2^{N-t+1}}$ which does not appear in the formula for any of the previous $f_i$; thus, let us set
$$h_{3(2^{t-1}-1)2^{N-t+1}+1}=n^{1+ t\epsilon}.$$

Its coefficient in formula (\ref{hf}) for $f_{i}$ is  $\binom{e-3(2^{t-1}-1)2^{N-t+1}}{e-i}$, for any
$$i= 3(2^{t-1}-1)2^{N-t+1},3(2^{t-1}-1)2^{N-t+1}+1,\dots, e,$$
and the maximum of these binomial coefficients is achieved for
$$i=e/2+3(2^{t-1}-1)2^{N-t}.$$

Therefore, if we pick
$$h_{3(2^{t-1}-1)2^{N-t+1}+2}=h_{3(2^{t-1}-1)2^{N-t+1}+3}=\dots =h_{e/2+3(2^{t-1}-1)2^{N-t}+2}=n,$$
then, similarly, we easily obtain  that $f_{e/2+3(2^{t-1}-1)2^{N-t}}$, i.e., $f_{3(2^{t}-1)2^{N-t}-2}$, yields the $t$-th peak  for $n\gg 0$, as  we wanted to show. 

Notice that we have defined the entire $h$-vector $h$ up to $h_{e+1}$, and that the $N$-th peak, corresponding to $f_{e-1}$, is indeed the last.

It remains to show that  $$h=(1,h_1=n^{1+ \epsilon},h_2=n,\dots,n,n^{1+ 2\epsilon},n,\dots,h_{e-2}=n,h_{e-1}=n^{1+ N\epsilon},h_{e}=n,h_{e+1}=n)$$
is  an $O$-sequence for some integer $n\gg 0$, for a suitable choice of $\epsilon$. Clearly, the only entries where Macaulay's growth condition is not trivially verified are  those where $n$ grows to some $n^{1+ t\epsilon}$. It is easy to see that, \emph{a fortiori}, it  suffices to control the growth of  $h_{e-2}$; namely, $n$ needs to satisfy
$$n^{1+ N\epsilon}\le n^{<e-2>}.$$

But it is a standard task to check that $n^{<e-2>}\ge n^{\frac{e-1}{e-2}}$ for $n$ large enough, using  the $(e-2)$-binomial expansion of $n$ and the fact that $\binom{m}{e-2}$ is asymptotic to $\frac{m^{e-2}}{(e-2)!}$, for $m$ large. This means that we want $\epsilon>0$ to satisfy $1+ N\epsilon \le \frac{e-1}{e-2}$, i.e.,
$$0<\epsilon \le \frac{1}{N(e-2)}=\frac{1}{6N(2^{N-1}-1)}.$$

Thus, for any (rational number) $\epsilon$ within the previous range, there exists an  $n(\epsilon)$ such that for all integers $n\ge n(\epsilon)$,  with $n^{\epsilon}$ integer, the corresponding $h$-vector is an $O$-sequence, and the theorem follows.
\end{proof}

We now provide a second, combinatorial proof of our main theorem in the case of pure $f$-vectors. In fact, our argument will show more: there exist nonunimodal pure $f$-vectors $f$ with any number of peaks and such that, roughly speaking, 	\emph{these peaks may appear in any degrees of our choice in the second half of $f$}. Since  the first half of a pure $f$-vector is always increasing (see e.g. \cite{hausel,hibi}, where this fact is proved in a more general context), our result is optimal in this regard.

\begin{theorem}
Fix any $N\ge 2$, and let $k_1, k_2,\dots,k_N=e+1$ be positive integers such that $k_{i+1}\ge k_i+2$ for all $i$, and $k_1\geq k_N/2$.
Then there exists a nonunimodal pure $f$-vector $f=(1,f_0,f_1,\dots ,f_e)$, such that $f$ has  a peak  at $f_{k_i-1}$, for each $i\ge 1$.

In particular, since a truncation of a pure $f$-vector is also pure, it follows that for any integer $N\ge 2$, there exists a nonunimodal pure $f$-vector having exactly $N$ peaks.
\end{theorem}

\begin{proof}
For any integer $s\ge k_N$, let $\Delta^{(s)}$ be the pure simplicial complex on $s$ elements whose facets are the subsets of cardinality $k_N$. Since $e=k_N-1$, define $\Delta$ as the $e$-dimensional pure complex obtained as the disjoint union of one copy of $\Delta^{(r)}$ and $a_i$ copies of $\Delta^{(2k_i)}$, for each $i=1, \dots, N-1$. We want to show that, for $r$ large enough, there exists a choice of $(a_1,\dots,a_{N-1})\in \mathbb N^{N-1}$ such that the $f$-vector $f=(1,f_0,\dots ,f_e)$ of $\Delta$ satisfies the hypotheses  of the statement.

From our construction it is clear that, for $m=0,\dots,e$,

$$ f_m=\sum_{i=1}^{N-1} a_i\binom{2k_i}{m+1}+\binom{r}{m+1}. $$

Fix $\alpha=1,\dots,N$. In order for $f_{k_{\alpha}-1}$ to be a peak of $f$, it is enough that it satisfies the following two inequalities:

$$ \sum_{i=1}^{N-1} a_i\binom{2k_i}{k_{\alpha}+1}+\binom{r}{k_{\alpha}+1}-\sum_{i=1}^{N-1} a_i\binom{2k_i}{k_{\alpha}}-\binom{r}{k_{\alpha}}<0; $$
$$ \sum_{i=1}^{N-1} a_i\binom{2k_i}{k_{\alpha}}+\binom{r}{k_{\alpha}}-\sum_{i=1}^{N-1} a_i\binom{2k_i}{k_{\alpha}-1}-\binom{r}{k_{\alpha}-1}>0.$$

Notice that, in fact, since $f_{k_N-1}$ is the last entry of  $f$, we may discard the first inequality for $f_{k_N-1}$. Similarly, it is easy to see from our construction  that $f$ increases until $f_{k_1-1}$; thus, we may discard the second inequality for $f_{k_1-1}$.

Using the identity $\binom{n}{k} - \binom{n}{k-1}=\binom{n+1}{k} \frac{n+1-2k}{n+1}$, we  rewrite the two previous inequalities as:

\Small
$$ p_1^{(\alpha)}=p_1^{(\alpha)}(a_1,\dots,a_{N-1},r)=\sum_{i=1}^{N-1}a_i\binom{2k_i+1}{k_{\alpha}+1}\frac{2k_i+1-2(k_{\alpha}+1)}{2k_i+1} + \binom{r+1}{k_{\alpha}+1}\frac{r+1-2(k_{\alpha}+1)}{r+1}<0;$$
\normalsize
$$ p_2^{(\alpha)}=p_2^{(\alpha)}(a_1,\dots,a_{N-1},r)=\sum_{i=1}^{N-1}a_i\binom{2k_i+1}{k_{\alpha}}\frac{2k_i+1-2k_{\alpha}}{2k_i+1} + \binom{r+1}{k_{\alpha}}\frac{r+1-2k_{\alpha}}{r+1}>0.$$

Thus, we want to show that, for $\alpha=1,\dots,N$, this system of inequalities has a common solution in $(a_1,\dots,a_{N-1},r)\in \mathbb N^N$, for any given $k_1,\dots,k_N$ as in the statement.

By definition of the $k_i$, $\frac{2k_i+1-2k_{\alpha}}{2k_i+1}<0$ if and only if $\alpha>i$, and $\frac{2k_i+1-2(k_{\alpha}+1)}{2k_i+1}<0$ if and only if $\alpha \geq i$. Hence, for each $i$, the coefficients of $a_i$  in $p_2^{(\alpha+1)}$ and  $p_1^{(\alpha)}$  have the same sign.

For $\beta=1,\dots,N-1$,  define $f_1^{(\alpha,\beta)}$ (for $\alpha=1,\dots, N-1$) and $f_2^{(\alpha,\beta)}$ (for $\alpha=2,\dots, N$) as:

$$ f_1^{(\alpha,\beta)}=f_1^{(\alpha,\beta)}(a_1,\dots,a_{N-1}, r)=p_1^{(\alpha)}\frac{2k_\beta +1}{2k_\beta +1-2(k_\alpha +1)}\frac{1}{\binom{2k_\beta +1}{k_\alpha +1}}-a_\beta; $$
$$ f_2^{(\alpha,\beta)}=f_2^{(\alpha,\beta)}(a_1,\dots,a_{N-1}, r)=p_2^{(\alpha)}\frac{2k_\beta +1}{2k_\beta +1 -2k_\alpha}\frac{1}{\binom{2k_\beta +1}{k_\alpha}}-a_\beta.$$

It is easy to check that neither $f_1^{(\alpha,\beta)}$ nor $f_2^{(\alpha,\beta)}$  depends on $a_\beta$, and that, therefore,  the  conditions  $p_1^{(\alpha)} <0$ and $p_2^{(\alpha)}>0$ are equivalent to the following  two systems of $N-1$ simultaneous inequalities, which we will solve for $a_{\beta}$: the system $A_{\beta}$,  given by

\Small
$$ a_{\beta}<f_2^{(N,\beta)}; {\ }{\ }a_{\beta}<f_2^{(N-1,\beta)}; {\ }{\ }\dots ; {\ }{\ } a_{\beta}<f_2^{(\beta +1,\beta)};{\ }{\ }  a_{\beta}<-f_1^{(\beta -1,\beta)}; {\ }{\ }a_{\beta}<-f_1^{(\beta -2,\beta)};{\ }{\ } \cdots ;{\ }{\ } a_{\beta}<-f_1^{(1,\beta)} ;$$
\normalsize
and the system $B_{\beta}$, given by

\Small
$$ a_{\beta}>f_1^{(N-1,\beta)}; {\ }{\ }a_{\beta}>f_1^{(N-2,\beta)}; {\ }{\ }\dots ; {\ }{\ }a_{\beta}>f_1^{(\beta,\beta)} ;{\ }{\ } a_{\beta}>-f_2^{(\beta,\beta)};{\ }{\ }a_{\beta}>-f_2^{(\beta-1,\beta)};{\ }{\ } \dots ; {\ }{\ }a_{\beta}>-f_2^{(2,\beta)} .$$

\normalsize
Notice that, since $k_{\alpha +1}> k_{\alpha}+1$,  $\binom{r+1}{k_{\alpha+1}}\frac{r+1-2k_{\alpha+1}}{r+1}$ has a greater order of magnitude than $\binom{r+1}{k_{\alpha}+1}\frac{r+1-2(k_{\alpha}+1)}{r+1}$, for $r\gg 0$. This means that the term involving $r$ in $p_2^{(\alpha+1)}$ is larger than the corresponding term in $p_1^{(\alpha)}$; similarly,   for $r\gg 0$, the term involving $r$ in $f_2^{(\alpha +1,\beta)}$ is larger than the corresponding term in $f_1^{(\alpha,\beta)}$.

It is easy to see that if $f_2^{(\alpha +1,\beta)}$ is in the system $A_{\beta}$, then $f_1^{(\alpha,\beta)}$ is in the system $B_{\beta}$, both with positive signs. Similarly $f_1^{(\alpha,\beta)}$  in $A_{\beta}$ implies $f_2^{(\alpha+1,\beta)}$ in $B_{\beta}$, both with negative signs. Hence, for  $r\gg 0$, all terms involving $r$ in the inequalities of the system $A_{\beta}$ have a greater order of magnitude  than all of those in the system $B_{\beta}$. 

It follows that, for $r\gg 0$, the largest integer solution $a_\beta$ to $A_{\beta}$ is at least equal to the smallest integer solution $a_\beta$ to  $B_\beta$. Also, $f_2^{(N,\beta)}$ is in $A_\beta$ for all $\beta$, and for $r\gg 0$, its term involving $r$ has a greater order of magnitude than any other term. Hence, the largest integer solution $a_\beta$ to $A_\beta$ is positive.

Thus, if we choose $r$ large enough so that, for all $i=1,\dots,N-1$, the largest solution  $a_i\in \mathbb N$ to the system $A_i$  is at least equal to the smallest solution $a_i\in \mathbb N$  to the system $B_i$, then we  have clearly determined a tuple $(a_1,\dots,a_{N-1},r)\in \mathbb N^N$ that solves  our  initial system of inequalities. This completes the proof of the theorem.
\end{proof}

\section{The failing of the Interval Property}

In this final section, we provide an elegant, infinite family of counterexamples to the \emph{Interval Property} (IP) for pure $f$-vectors of dimension 2. The IP was first introduced in \cite{Za1}, where the second author conjectured it for the set of level and Gorenstein Hilbert functions (see \cite{Za1} for the relevant definitions).  Namely, the IP says that if $S$ is a given class of positive integer sequences, and $h,h'\in S$  coincide in all entries but one, say $h=(h_0,\dots ,h_{i-1},h_i,h_{i+1},\dots )$ and $h'=(h_0,\dots ,h_{i-1},h_i+\alpha ,h_{i+1},\dots )$ for some index $i$ and some positive integer $\alpha $, then $(h_0,\dots ,h_{i-1},h_i+\beta ,h_{i+1},\dots )$ is also in $S$, for each  $\beta =1,\dots , \alpha -1.$

As for level and Gorenstein Hilbert functions, the IP, which appears to be consistent with the main techniques used in that area, is still wide open. However, even though the IP also holds for many other important sequences in combinatorics and combinatorial commutative algebra, including for  Cohen-Macaulay $f$-vectors (see \cite{BMMNZ,St3} for details), most recently it has been disproved in a few interesting cases. Indeed, the IP was conjectured in \cite{BMMNZ} for pure $O$-sequences, and asked as a question (Question 9.4) for pure $f$-vectors (see also \cite{MNZ_exp}). As for pure $O$-sequences, in \cite{BMMNZ} the IP was proved in  degree 3, thus allowing a new approach to Stanley's matroid $h$-vector conjecture \cite{HSZ}. However, it was then disproved in large degree by A. Constantinescu and M. Varbaro \cite{CV}. Also, recently, R. Stanley and the second author showed the IP not to hold for arbitrary \emph{$r$-differential posets} (see \cite{RSFZ} for details), even though it is still open  for the main family of such posets, i.e., 1-differential posets (see also \cite{Byr}). 

In this section, we provide a simple argument that disproves the IP also for pure $f$-vectors. In fact, we even show the existence of a nice, infinite family of counterexamples in dimension 2, which is the smallest possible dimension where the IP might have failed. Our techniques are  constructive; in particular, they involve a suitable application of Steiner triple systems (some objects of design theory), and Stanley's characterization of Cohen-Macaulay $h$-vectors.

Recall that a  \emph{Steiner system} $S(l,m,r)$ is an $r$-element set $V$, together with a collection of $m$-subsets of $V$, called \emph{blocks}, such that every $l$-subset of $V$ is contained in exactly one block (see e.g. \cite{codi,liro}). The case $S(2,3,r)$ is that of Steiner \emph{triple} systems (STS) of \emph{order} $r$.

Clearly, since all blocks have the same cardinality, by identifying each element of $V$ with a variable $y_i$,  the existence of Steiner systems (and similarly for other designs) is equivalent to that of certain pure $f$-vectors. For instance,  the existence of $S(2,3,7)$ is tantamount to that of 7 squarefree degree 3 monomials  in $R=K[y_1,\dots,y_7]$, say $M_1,\dots,M_7$,  such that each squarefree degree 2 monomial of $R$ divides exactly one of the $M_i$. Since such monomials can  easily be shown to exist (uniquely, up to  isomorphism), then also a (unique) STS $S(2,3,7)$ does exist (called the \emph{Fano plane}, for geometric reasons).

In fact, it is trivial to see that if $S(2,3,r)$ exists, then $r$ is congruent to 1 or  3 modulo 6. A nice classical result of T.P. Kirkman \cite{kirk} is then that this condition is also sufficient. In other words, we have that
$$f=\left(1,r,\binom{r}{2}, \binom{r}{2}/3 \right)$$
is a pure $f$-vector if and only if $r$ is congruent to 1 or  3 modulo 6.

\begin{theorem}\label{ip} The Interval Property fails for the set of pure $f$-vectors. Namely, fix any integer $r$ congruent to 1 or 3 modulo 6, $r\ge 7$. Then:
$$f=\left(1,r,\binom{r}{2},\binom{r}{2}/3\right)$$ is a pure $f$-vector;
$$f'=\left(1,r,\binom{r}{2}-1,\binom{r}{2}/3\right)$$ is \emph{not} a pure $f$-vector; and
$$f''=\left(1,r,b,\binom{r}{2}/3\right)$$ is a Cohen-Macaulay, hence  pure, $f$-vector for all  $b=b_0, b_0+1, \dots, \frac{(r-1)(r+6)}{6}$, where $b_0$ is the smallest integer such that $h=\left(1,r-3,b_0-2r+3,\frac{(r-1)(r+6)}{6}-b_0\right)$ is an $O$-sequence. In particular,  since $0<b_0\le \frac{(r-1)(r+6)}{6}<\binom{r}{2}-1,$  the IP is disproved.
\end{theorem}

\begin{proof} That $f=\left(1,r,\binom{r}{2},\binom{r}{2}/3\right)$ is a pure $f$-vector was shown in the discussion before the theorem, since $r$ is  congruent to 1 or 3 modulo 6, and an STS $S(2,3,r)$ does exist under these hypotheses.

That $f''=\left(1,r,b,\binom{r}{2}/3\right)$ is a Cohen-Macaulay $f$-vector for all values of $b$ as in the statement can easily be verified, using Stanley's characterization of the $h$-vectors of Cohen-Macaulay  complexes (Lemma \ref{stan}) along with the linear transformation (\ref{hf}). 

Proving that $0<b_0\le \frac{(r-1)(r+6)}{6}<\binom{r}{2}-1$, for $r$ as in the statement, is also a standard exercise (in fact, we always have $b_0< \frac{(r-1)(r+6)}{6}$). Thus, it remains to show that $f'=\left(1,r,\binom{r}{2}-1,\binom{r}{2}/3\right)$ is not a pure $f$-vector. A simple elementary argument is 	the following.

Suppose  $f'$ is pure. Then there exist  $t=\binom{r}{2}/3$ squarefree degree 3  monomials in $R=K[y_1,\dots,y_r]$, say $M_1, \dots, M_t$, such that all but one of the squarefree degree 2  monomials of $R$ divide some $M_i$. Let this monomial be $y_1y_2$. Since the $M_i$ clearly have a total of $\binom{r}{2}$ degree 2 divisors (counted with multiplicity), each $y_jy_k$ appears exactly once as a factor of some $M_i$, with the only exceptions of $y_1y_2$, which does not appear, and one other squarefree degree 2  monomial, appearing twice.

Without loss of generality, we can assume  this latter  monomial is either $y_1y_3$ or $y_3y_4$. In either case, it follows that no two $M_i$ that are divisible by $y_2$ can have another  variable in common. Therefore, since $y_1y_2$ divides no $M_i$, the number of $M_i$ divisible by $y_2$ is $(r-2)/2$, which is impossible, since $r$ is odd. This completes the proof of the theorem.
\end{proof}

\begin{example} The smallest counterexample to the IP given by Theorem \ref{ip} is when $r=7$. Namely, $(1,7,21,7)$ is a pure $f$-vector (it corresponds to the Steiner triple system, or Fano plane, $S(2,3,7)$), and so are the Cohen-Macaulay $f$-vectors $(1,7,12,7)$ and $(1,7,13,7)$.

However, the $f$-vector $(1,7,20,7)$ is {not} pure, thus violating the IP. In fact, it is a simple exercise to check that $(1,7,b,7)$ is a pure $f$-vector if and only if $b\in [12,19]\cup \{21\}$. See \cite{CKK} for a generalization of this fact.
\end{example}

\section*{Acknowledgements} We wish to thank the referee for comments, and Don Kreher for an interesting discussion of design theory, which has proved very helpful in showing that $\left(1,r,\binom{r}{2}-1,\binom{r}{2}/3\right)$ is not a pure $f$-vector. He, C.J. Colbourn and M.S. Keranen \cite{CKK} (personal communication) have then been working on a characterization of the pure $f$-vectors of the form $f=(1,r,b,c)$.



\begin{thebibliography}{ll}
\bibitem{BMMNZ} M. Boij, J. Migliore, R. Mir\`o-Roig, U. Nagel and   F. Zanello: ``On the shape of a pure $O$-sequence'', Mem. Amer. Math. Soc. \textbf{218} (2012), no. 2024, vii + 78 pp.. 
\bibitem{BH}  W. Bruns and J. Herzog: ``Cohen-Macaulay rings'', Cambridge Studies in Advanced Mathematics, No. 39, Revised Edition, Cambridge University Press, Cambridge, U.K. (1998).
\bibitem{Byr} P. Byrnes:  Ph.D. Thesis, University of Minnesota, in preparation.
\bibitem{codi} C.J. Colbourn and J.H. Dinitz, Eds.: ``Handbook of Combinatorial Designs,'' CRC Press, Boca Raton, FL (1996).
\bibitem{CKK} C.J. Colbourn, M.S. Keranen and D.L. Kreher: \emph{$f$-vectors of pure complexes of rank three}, preprint.
\bibitem{CV} A. Constantinescu and M. Varbaro: \emph{$h$-vectors of matroid complexes}, preprint. Available on the \href{http://arxiv.org/abs/1212.3226}{arXiv}.
\bibitem{HSZ} H.T. H\`a, E. Stokes and F. Zanello: {\em Pure $O$-sequences and matroid $h$-vectors}, Ann. Comb. \textbf{17} (2013), no. 3, 495--508.
\bibitem{hausel} T. Hausel: \emph{Quaternionic geometry of matroids}, Cent. Eur. J. Math. {\bf 3} (2005), no. 1, 26--38.
\bibitem{hibi} T. Hibi: \emph{What can be said about pure $O$-sequences?}, J. Combin. Theory Ser. A {\bf 50} (1989), no. 2, 319--322.
\bibitem{Hu} J. Huh: \emph{Milnor numbers of projective hypersurfaces and the chromatic polynomial of graphs}, J. Amer. Math. Soc. \textbf{25} (2012) 907--927.
\bibitem{huh2} J. Huh: \emph{$h$-vectors of matroids and logarithmic concavity}, preprint. Available on the \href{http://arxiv.org/abs/1201.2915}{arXiv}.
\bibitem{HK} J. Huh and E. Katz: \emph{Log-concavity of characteristic polynomials and the Bergman fan of matroids}, Math. Ann. \textbf{354} (2012), 1103--1116.
\bibitem{kirk} T.P. Kirkman: \emph{On a Problem in Combinatorics}, Cambridge Dublin Math. J. \textbf{2} (1847), 191--204.
\bibitem{lenz} M. Lenz: \emph{The $f$-vector of a realizable matroid complex is strictly log-concave}, Comb. Probab. and Computing, to appear. Available on the  \href{http://arxiv.org/abs/1106.2944}{arXiv}.
\bibitem{liro} C.C. Lindner and C.A. Rodger: ``Design Theory,'' CRC Press, Boca Raton, FL (1997).
\bibitem{Ma} F.H.S. Macaulay: \emph{Some properties of enumeration in the theory of modular systems}, Proc. London Math. Soc. \textbf{26} (1927), 531--555.
\bibitem{MNZ_exp} J. Migliore,  U. Nagel and   F. Zanello: \emph{Pure $O$-sequences: known results, applications and open problems}, in: ``Commutative Algebra.  Expository Papers Dedicated to David Eisenbud on the Occasion of His 65th Birthday'' (I. Peeva, Ed.), Springer (2013). Available on the \href{http://arxiv.org/abs/1204.5247}{arXiv}.
\bibitem{MS} E. Miller and B. Sturmfels: ``Combinatorial commutative algebra'', Graduate Texts in Mathematics \textbf{227}, Springer-Verlag, New York (2005).
\bibitem{St1} R. Stanley: \emph{Cohen-Macaulay Complexes}, in ``Higher Combinatorics'' (M. Aigner, Ed.), Reidel, Dordrecht and Boston (1977), 51--62.
\bibitem{St3} R. Stanley: ``Combinatorics and commutative algebra'', Second Ed., Progress in Mathematics \textbf{41}, Birkh\"auser Boston, Inc., Boston, MA (1996).
\bibitem{RSFZ} R. Stanley and F. Zanello: \emph{On the rank function of a differential poset}, Electron. J. Combin. \textbf{19} (2012), no. 2, P13, 17 pp..
\bibitem{Za0} F. Zanello: \emph{A non-unimodal codimension 3 level $h$-vector}, J. Algebra \textbf{305} (2006), no. 2, 949--956.
\bibitem{Za1} F. Zanello: \emph{Interval Conjectures for level Hilbert functions}, J. Algebra \textbf{321} (2009), no. 10, 2705--2715.
\end{thebibliography}
\end{document}